\documentclass[12pt]{amsart}
\usepackage{amsmath,amsfonts,amssymb,amsthm,amscd, latexsym, euscript, xypic, verbatim}
\usepackage[all]{xy}
\makeatletter
\def\section{\@startsection{section}{1}%
  \z@{1.1\linespacing\@plus\linespacing}{.8\linespacing}%
  {\normalfont\Large\scshape\centering}}
\makeatother

\theoremstyle{plain}

\newtheorem*{conj*}{Root Groups Conjecture}
\newtheorem*{thm1.2}{(1.2) Theorem}
\newtheorem*{thm1.3}{(1.3) Theorem}
\newtheorem*{thm1.4}{(1.4) Theorem}
\newtheorem*{prop*}{Proposition}
\newtheorem{Thm}{Theorem}
\newtheorem{Prop}[Thm]{Proposition}

\newtheorem{prop}{Proposition}[section]

\newtheorem{thm}[prop]{Theorem}
\newtheorem{cor}[prop]{Corollary}
\newtheorem{lemma}[prop]{Lemma}

\newtheorem{term}[prop]{Terminology}

\theoremstyle{definition}

\newtheorem{Defi}[Thm]{Definition}

\newtheorem*{Def*}{Definition}

\newtheorem{example}[prop]{Example}

\newtheorem{notation}[prop]{Notation}
\newtheorem*{notation*}{Notation}

\newcommand{\calc}{\mathcal{C}}

\newcommand{\calk}{\mathcal{K}}

\newcommand{\calz}{\mathcal{Z}}

\newcommand{\zz}{\mathbb{Z}}

\newcommand{\ga}{\alpha}
\newcommand{\gb}{\beta}
\newcommand{\gc}{\gamma}
\newcommand{\gC}{\Gamma}

\newcommand{\gD}{\Delta}

\newcommand{\gL}{\Lambda}

\newcommand{\gvp}{\varphi}
\newcommand{\gr}{\rho}
\newcommand{\gs}{\sigma}

\newcommand{\lan}{\langle}
\newcommand{\ran}{\rangle}

\newcommand{\onto}{\twoheadrightarrow}

\newcommand{\half}{\textstyle{\frac{1}{2}}}

\newcommand{\widebar}[1]{\overset{\mskip1mu\hrulefill\mskip1mu}{#1}
                \vphantom{#1}}

\newcommand{\W}{\widebar}

\newcommand{\cl}{\gC^{\gvp}}

\newcommand{\invlim}{\varprojlim}

\numberwithin{equation}{section}

\begin{document}
\title[Subnormal  closure]{Subnormal closure  of a homomorphism}
\author[Emmanuel D.~Farjoun\quad Yoav Segev]
{Emmanuel~D.~Farjoun\qquad Yoav Segev}

\address{Emmanuel D.~Farjoun\\
        Department of Mathematics\\
        Hebrew University of Jerusalem, Give Ram\\
        Jerusalem 91904\\
        Israel}
\email{farjoun@math.huji.ac.il}

\address{Yoav Segev\\
        Department of Mathematics\\
        Ben Gurion University\\
        Beer Shea 84105\\
        Israel}
\email{yoavs@math.bgu.ac.il}

\begin{abstract}
Let $\gvp\colon\gC\to G$ be a homomorphism of groups.  In this paper we introduce
the notion of a subnormal map (the inclusion of a subnormal subgroup into a group
being a basic prototype).  We then consider factorizations $\gC\xrightarrow{\psi} M\xrightarrow{n} G$
of $\gvp,$ with $n$ a subnormal map. We search for a universal such factorization.
When $\gC$ and $G$ are finite we show that such universal factorization exists: $\gC\to\gC_{\infty}\to G,$
where $\gC_{\infty}$ is a hypercentral extension of the subnormal closure $\calc$
of $\gvp(\gC)$ in $G$ (i.e.~the kernel of the extension $\gC_{\infty}\onto\calc$ is contained
in the hypercenter of $\gC_{\infty}$).  This is closely related to the a relative version of the Bousfield-Kan $\zz$-completion tower
of a space.
The group $\gC_{\infty}$  is  the inverse limit of the normal closures tower of $\gvp$ introduced
by us in a recent paper. We prove several stability and  finiteness properties of the tower and its inverse
limit $\gC_\infty$.
\end{abstract}

\date{\today}
\maketitle
 
\section{Introduction}

Throughout this note $\gvp\colon\gC\to G$ is a homomorphism of groups.
In a previous  paper \cite{FS1}  we considered
the notion of the {\it free normal closure} $\cl$ of $\gvp,$ (related to \cite{BH, BHS}) and the {\it free normal closures tower}
$\{\gC_{i,\gvp}\}_{i=1}^{\infty}$ of $\gvp$ (see equation \eqref{eq tower of fnc} below).
In this paper we study the behavior and the properties of the inverse limit
$\gC_{\infty,\gvp}$ of the tower of free normal closures
of $\gvp$.  This tower generalizes and connects  the quotients of the lower central series
$\gC/\gc_i(\gC),\ i=1, 2, 3,\dots,$   gotten here for $G=1,$ and the descending series
of successive normal closures of a subgroup $H$ of $G$ (see \S2).  Notice that
in the case $G=1,$ the group $\gC_{\infty,\gvp}$ is the nilpotent completion of $\gC$.
Thus for an arbitrary map $\gvp,$ the group $\gC_{\infty,\gvp}$ can be thought of as a
relative nilpotent completion associated with a homomorphism rather than with a group.
Some of the results here support this point of view.

Let us recall from \cite[section 4]{FS1}
that the normal closures tower (we often omit the word ``free'')
associated with a homomorphism  $\gvp$ is a tower of groups as follows:

\begin{equation}\label{eq tower of fnc}
\xymatrix{
\gC\ar[dr]^>>>{\gvp_{k+1}}\ar@/^1ex/[drr]^{\gvp_k}\ar@/^2ex/[drrrr]^{\gvp_2}\ar@/^4ex/[drrrrr]^{\gvp_1}\\
\dots\ar[r]  &\gC_{k+1}\ar[r]^<<{\scriptscriptstyle{\W{\gvp_k}}} &
\gC_k\ar[r] &
\dots \ar[r] & \gC_2\ar[r]^<<{{\scriptscriptstyle{\W{\gvp_1}}}} &\gC_1=G\\
\\
}
\end{equation}

\noindent
where $\W{\gvp_i}$ is a normal map (our terminology for  a crossed module),
and where $\gC_{i+1}=\gC_{i+1,\,\gvp}$ is the (free) normal closure $\gC_i^{\gvp_i}$
of the map $\gC\xrightarrow{\gvp_i}\gC_i,$ for all $i\ge 1$.
The group $\gC_{\infty}=\gC_{\infty,\gvp},$ the stability of the tower $\{\gC_{i,\gvp}\}$ and its relation to the groups $\gC, G$ and the map $\gvp,$
are thus the main topics  of the present study.

One way to think about the above tower is that it represents an attempt  to factor the map $\gvp,$ in a universal way,
into a composition of ``simpler maps'' (i.e.~normal maps).  This, of course, cannot
be done in general, and $\gC_{\infty,\gvp}$ is a kind of ``hybrid'' of $\gC$ and
$G$ giving a factorization $\gC\to\gC_{\infty,\gvp}\to G$.  Passing to topological spaces via the classifying space
construction we get a map $B\gC\to BG.$ The present result  shows that for finite groups $\gC, G$ this map has a finite "relative" Bousfield-Kan
 tower of principal fibrations \cite{BK} whose fibres are, in general,  neither connected nor nilpotent groups:
\[
B\gC_\infty=B\gC_n\to B\gC_{n-1}\to\cdots \to B\gC_2\to BG;
\]
with a terminal term $B\gC_\infty=B\gC_n\to BG$ being  {\em the universal space,} under $B\gC$ for which there is such a tower of principal fibrations.
This raises the question of finding such universal decompositions of more general maps of spaces $X\to Y.$
Notice that we have an induced map $\gvp_{\infty}\colon\gC\to\gC_{\infty};$ we often ask how far is this map from an isomorphism.

One quick corollary of our  main results concerns a map of nilpotent groups:

\begin{Thm}\label{thm nil}
Let $\gvp\colon\gC\to G$ be a homomorphism of nilpotent groups, then $\gvp_{\infty}\colon\gC\to\gC_{\infty}$
is an isomorphism.
\end{Thm}

Theorem \ref{thm nil} is Corollary \ref{cor invlim nil}(2).  This corollary can be viewed as a version
for any map of nilpotent groups, of the
well known property that any subgroup of a nilpotent group is subnormal, i.e.~equal to its subnormal closure.

In this spirit we have for any homomorphism of finite groups:

\begin{Thm}\label{Thm nil quotients}
Let $\gvp\colon\gC\to G$ be a homomorphism of finite groups.  Then $\gvp_{\infty}$
induces an isomorphism of the descending central series quotients, $\gC/\gc_i(\gC)\cong\gC_{\infty}/\gc_i(\gC_{\infty}),$
for all $i\ge 1$.
\end{Thm}

Theorem \ref{Thm nil quotients} is proved in Proposition \ref{prop nil resid}. We note that by previous results, this is certainly not true for a
general map of groups, e.g., when $\Gamma$ is a free group with infinitely   many generators and $G=1.$

Our next result extends a partial result in   \cite{FS1} to the general finite case. An estimate of the size of
$\Gamma_\infty$ is given  below in Theorem \ref{thm finite}.
\begin{Thm}\label{Thm finite}
If  $\gvp\colon\gC\to G$ be a homomorphism of finite groups then  $\gC_{\infty,\gvp}$ is  a finite group.
\end{Thm}

\noindent
We note that even when $\gvp$ is the trivial map between two finite cyclic groups,
the groups $\Gamma_{i,\gvp}$ grow indefinitely in size with $i,$
but their inverse limit is finite, and is isomorphic  to the domain in this case (Theorem \ref{thm nil}).

Recall that for a finite group $G$ and
a subgroup $H\le G,$ the {\it subnormal closure} of $H$ in $G$ is the smallest subnormal
subgroup of $G$ containing $H$. Thus Theorem \ref{Thm finite} is again an extension to any map of finite groups, of the  trivial observation that the subnormal closure of a subgroup of a finite group is well defined and finite.
This result complements the dual result proved in  \cite{FS1} for the tower
of injective normalizers of a map of finite groups.

The next theorem characterizes, for maps of finite groups, the factorization $\gC\to \gC_\infty \to G$  as a universal one among all subnormal factorizations (see below) the proof is given in \S\ref{sect universal}:

The following is a basic definition:

\begin{Defi}
A {\it subnormal map} $n\colon M\to G,$ is a homomorphism such that there exists a
{\it finite} series of normal maps $n_i\colon M_{i+1}\to M_i,\ 1\le i\le k,$ with $M_1=G,$
whose composition is $n$.
\[
\xymatrix{
M_{k+1} \ar[r]^{n_k}\ar@/_3ex/[rrrr]_{n}& M_{k}\ar[r]&\dots\ar[r] &M_2\ar[r]^{n_1} &G.
}
\]
\noindent
Notice that  the image subgroup $n(M_{k+1})$ is subnormal in $G$.
\end{Defi}

For example, any map of nilpotent groups is a subnormal map.
\medskip

\begin{Thm}\label{thm universal}
If  $\gvp\colon\gC\to G$ be a homomorphism of finite groups then the factorization   $\gC\to \gC_{\infty,\gvp} \to G$ is  the universal initial
factorization of $\gvp,$ among all factorizations $\gC\to S\to G$ of $\gvp,$ 
with the right map $S\to G$ being a subnormal map; namely, it maps uniquely to every such factorization
via $\gC_{\infty,\gvp}\to S$.
\end{Thm}
\bigskip

\subsection*{Main tools} 
Here we sum-up the main technical tools for proving the above results. They consist of showing that the
limiting group $\gC_{\infty,\gvp}$ does not change up to a canonical isomorphism, when one changes the domain or range of $\gvp$
in certain controlled ways.

First we consider changing the range via factoring our map $\gvp$ through any subnormal map $M\to G.$

The following theorem is one of our main tools showing we can perform the above mentioned replacement:

\begin{Thm}\label{thm subnormal factorization}
Let $\gvp\colon\gC\to G$ be a homomorphism and let $\gC\xrightarrow{\psi} M\xrightarrow{n} G$
be a factorization of $\gvp$ such that $n\colon M\to G$ is a subnormal map.
Then $\gC_{\infty,\gvp}\cong\gC_{\infty,\psi}.$
\end{Thm}

\noindent
Theorem \ref{thm subnormal factorization} is proved right after Proposition \ref{prop same inv lim}.
\medskip

The next result shows that one can factor out
a certain portion $\calk$ of the kernel $K$ of $\gvp$ and
still obtain that $\gC_{\infty,\gvp}\cong(\gC/\calk)_{\infty,\gr},$ where
$\gr\colon\gC/\calk\to G$ is the map induced by $\gvp$.

To define $\calk,$ define the {\it descending series of
successive commutators of $K$ with $\gC$} by $K_1=K,\, K_2=[\gC,K],$ and
in general $K_{i+1}=[\gC,\, K_i]$.  If this series terminates after a
finite number of steps, we let $\calk$ be the final member of this series.

\begin{Prop}\label{prop Kf}
Let $\gvp\colon\gC\to G$ be a homomorphism, and suppose that the
descending series of successive commutators of $K=\ker\gvp$ with $\gC$ terminates
after a finite number of steps.  Let $\calk$ be the terminal member of
this series.  Then
$\gC_{\infty,\gvp}\cong(\gC/\calk)_{\infty,\gr},$
where $\gr\colon\gC/\calk\to G$ is the map induced by $\gvp$.
\end{Prop}

Proposition \ref{prop Kf} is Proposition \ref{prop Kterm}(2).
Notice that  in the notation of Proposition \ref{prop Kf},
the kernel of $\gr$ is $K/\calk$ and hence $\ker\gr$ is contained
in some member of the ascending central series of $\gC/\calk$.
In Example \ref{eg G perfect} we show however that
the nilpotency class of $K/\calk$ cannot be bounded.

\section{Equivalence of normal closures towers}\label{}

In this section we prove the results of the introduction, and analyze further the group $\gC_{\infty,\gvp}$.
Our main tool is to compare the normal closures tower $\{\gC_{i,\gvp}\}$ to towers $\{\gL_{i,\gr}\},$
for various homomorphism $\gr\colon\gL\to H,$ which are, of course, related to $\gvp$.

As noted above, throughout this paper $\gvp\colon\gC\to G$ is a fixed homomorphism of groups.
We use  the following notation.  $K=\ker\gvp$ and $\cl$ is the  (free) normal closure of $\gvp$. This, we recall, 
is the universal factorization $\Gamma\xrightarrow{c_{\gvp}} \cl\xrightarrow{\W{\gvp}} G$  of $\gvp$  with the right map being normal.

As in diagram \eqref{eq tower of fnc}, $\{\gC_i\}=\{\gC_{i,\gvp}\}$ is the normal closures tower of $\gvp,$
and $\gvp_i,\, \W{\gvp_i}$ are as in diagram \eqref{eq tower of fnc}.
$\gC_{\infty}=\invlim\gC_i$ and $\gvp_{\infty}\colon\gC\to\gC_{\infty}$ is the natural map.

\begin{term}
Let $k\ge 0$ be an integer, and let $\{H_i\}_{i=k}^{\infty}$ be a decreasing or increasing series of groups.
We say that $\{H_i\}$ \texttt{terminates} if there exists an integer $t\ge k,$
such that $H_t=H_{t+1}=H_{t+2}=\dots$.  In this case we call $H_t$ the \texttt{terminal member}
of the series and say that the series terminates at $H_t.$
\end{term}

Recall from \cite[p.~385]{R} the notion of
the {\it series of successive normal closures--in the usual sense-- of a subgroup $H\le G$ in $G$}.
This is the decreasing series  defined by $C_0=G,\, C_1=\lan H^{C_0}\ran,$ and in general
$C_{i+1}=\lan H^{C_i}\ran$.

\begin{notation}\label{not Ci and Ki}
\begin{enumerate}
\item
The series  of successive normal closures of $\gvp(\gC)$ in $G$
will be denoted $G=C_0,\, C_1,\, C_2\, \dots$.  If this series
terminates, then we denote by $\calc$ its terminal member.

\item
 Let $K:=\ker\gvp$ and define the decreasing {\it  series  of
successive commutators of $K$ with $\gC$} by $K_1=K,\, K_2=[\gC,K],$ and
in general $K_{i+1}=[\gC,\, K_i]$. If this series terminates, then
we denote by $\calk$ its terminal member.
\end{enumerate}
\end{notation}

Next $H=\gc_1(H),\, \gc_2(H),\dots$ denotes the descending central series of the
group $H$.  If this series terminates, then $\gc_{\infty}(H)$ denotes the terminal member of this
series.
Finally, the ascending central series of $H$ is denoted $1=Z_0(H), Z_1(H),\dots,$ and the same
convention as above applies for the notation $Z_{\infty}(H)$.

We refer the reader to \cite[section 2]{FS1} for the notions of a {\it normal map} and
of {\it normal morphism,} where references to previous work on this subject in given.  In \cite[section 3]{FS1} the reader will find some basic properties
of $\cl$ and in \cite[section 4]{FS1} some basic properties of the normal closures tower
of $\gvp$.

Let us start by recalling the naturality of the normal closures tower.

\begin{lemma}\label{lem maps between invlimits}
Any commutative diagram
\begin{equation}\label{eq naturality}
\xymatrix{
\gC\ar[r]^{\gvp}\ar[d]_{\mu} & G\ar[d]^{\eta}\\
\gC'\ar[r]^{\gvp'} &G'\\
}
\end{equation}

induces a commutative diagram between the towers of normal closures of $\gvp$ and $\gvp'$:

Namely if we set $\gC_1:=G,\ \gC_1'=G'$ and $\eta_1:=\eta,$ we have: 

\begin{equation}\label{eq maps between invlimits}
\xymatrix{
\Gamma\ar[r]\ar[d]^\mu &\gC_{i+1}\ar[r]^{\W{\gvp_i}}\ar[d]_{\eta_{i+1}} &\gC_i\ar[r]\ar[d]^{\eta_i} &\dots\ar[r] &\gC_1\ar[d]^{\eta_1}\\
\Gamma'\ar[r] &\gC_{i+1}'\ar[r]^{\W{\gvp_i'}} &\gC_i'\ar[r]&\dots\ar[r] &\gC_1'
}
\end{equation}
  Thus there is a canonical map
$\eta_{\infty}\colon\gC_{\infty}\to\gC'_{\infty}$ and equalities: $\gvp_i\circ\eta_i=\mu\circ\gvp_i',$ for all
$i\ge 1$.
\end{lemma}
 
\begin{proof}
Take in diagram (2.2) of \cite{FS1}, $M'=\gC_2',\ \psi'=\gvp_2'$ and $n'=\W{\gvp_1'},$ and use the universality properties
of $\gC_2$ to obtain $\eta_2$.  Then proceed in this manner replacing each time
$\gC_i,\ \gC_i',$ with $\gC_{i+1},\ \gC_{i+1}'$ respectively.
\end{proof}

The following Proposition will be used
in the proof of Theorem \ref{thm subnormal factorization}
of the introduction. It addresses the replacement of $G$ by the domain of any  subnormal map $M\to G$ that
factors our map $\gvp$.  In that case the map of towers \eqref{eq map of towers} is in fact a pro-isomorphism of towers:

\begin{prop}\label{prop same inv lim}
Let $\gC\xrightarrow{\psi}M\xrightarrow{n} G$ be a factorization of $\gvp$ with $n$ a normal map.
Then the above \ref{eq maps between invlimits} commutative diagram extends 
to  a natural commutative diagram in which the upper (resp.~lower) raw  is the normal closures tower of $\Gamma\xrightarrow{\psi} M$
 (resp.~$\Gamma\xrightarrow{\gvp} G$):
\begin{equation}\label{eq map of towers}
\xymatrix{
 \Gamma\ar[rr]\ar[d]^=   &&\gD_i\ar[d]^{\mu_i}\ar[r]^{\W{\psi}_{i-1}} &\gD_{i-1}\ar[d]^{\mu_{i-1}}\ar[r] &\dots\ar[r] & \gD_3\ar[r]^{\W{\psi}_2}\ar[d]^{\mu_3} &  \gD_2\ar[r]^{\W{\psi}_1}\ar[d]^{\mu_2} &\gD_1=M\ar@<-2.5ex>[d]^{\mu_1=n}\\
\Gamma
\ar[r]   &\gC_{i+1}\ar[r]_{\W{\gvp}_i}\ar[ru]_{\gr_i} &\gC_i\ar[r]_{\W{\gvp}_{i-1}}\ar[ru]_{\gr_{i-1}} &\gC_{i-1}\ar[r] &\dots \ar[r] &\gC_3\ar[r]_{\W{\gvp}_2}\ar[ru]_{\gr_2} &\gC_2\ar[r]_{\W{\gvp}_1}\ar[ru]_{\gr_1} &\gC_1=G
}
\end{equation}
\end{prop}
\begin{proof}
We need the following lemma which is an obvious analog of the similar situation
for two normal subgroup of the same group:

\begin{lemma}\label{lem normal morphism}
Consider the commutative diagram
\[
\xymatrix{
M_1\ar[rr]^{\mu}\ar[rd]_{n_1} && M_2\ar[ld]^{n_2}\\
&G
}
\]
where $n_i$ are normal maps, for $i=1, 2,$ and $\mu$ is a normal morphism.
\end{lemma}
\begin{proof}
Define an action of $M_2$ on $M_1$ by $m_1^{m_2}=m_1^{n_2(m_2)},\ m_i\in M_i$.  Then it is immediate that
$\mu$ becomes a normal map with the above normal structure.
\end{proof}

As usual let $\gC_1:=G$ and $\gvp_1:=\gvp$.  Consider the normal closures tower corresponding to $\psi$.
So the groups $\gC_i$ in diagram \eqref{eq tower of fnc} are replace by $\gD_i$ and
the maps $\gvp_i,\ \W{\gvp}_i$ are replaced by $\psi_i\ \W{\psi}_i,$ respectively.
Thus $\gD_1=M$ and $\psi_1=\psi$.
Set $\mu_1=n,$ thus $\mu_1\colon\gD_1\to\gC_1$ is a normal map.

We show that there are normal maps $\mu_i, \gr_i,$ $i\ge 1,$ as in
diagram \eqref{eq map of towers} such that 
\begin{equation}\label{eq maps}
\gvp_{i+1}\circ \gr_i=\psi_i\text{  and }\psi_i\circ\mu_i=\gvp_i,\text{ for all }i\ge 1.
\end{equation}

\noindent
Now since $\gC\xrightarrow{\psi_1}\gD_1\xrightarrow{\mu_1} \gC_1$ is a factorization of $\gvp_1,$
by the universality property of $\gC_2,$ there
exists a normal morphism $\gr_1\colon\gC_2\to\gD_1$ for the lower right triangle, and such that
$\gvp_2\circ\gr_1=\psi_1$. By Lemma \ref{lem normal morphism}, $\gr_1$ is a normal map.

Let now $i\ge 2,$ and suppose that $\gr_{i-1}$ and $\mu_{i-1}$ were defined,
they are normal maps, diagram \eqref{eq map of towers} is commutative up to the
the $(i-1)$-step, and equation \eqref{eq maps} holds.
By equation \eqref{eq maps} we have $\gvp_i\circ \gr_{i-1}=\psi_{i-1}$.
By  the universality of $\gD_i,$ and since $\gr_{i-1}$ is a normal map, there exists a map
$\mu_i\colon \gD_i\to \gC_i$ such that
\begin{itemize}
\item[(a)] $\psi_i\circ\mu_i=\gvp_i.$

\item[(b)]
$\mu_i\circ \gr_{i-1}=\W{\gvp}_i,$ and $\mu_i$ is a normal morphism for the right triangle.
\end{itemize}

\noindent
By Lemma \ref{lem normal morphism}, $\mu_i$ is a normal map.

Next, by the universality property of $\gC_{i+1},$ by (a), and since $\mu_i$ is a normal
map, there exists a normal morphism $\gr_i\colon\gC_{i+1}\to\gD_i$ for the lower right triangle,
this triangle is commutative ($\gr_i\circ\mu_i=\W{\gvp}_i$) and  $\gvp_{i+1}\circ\gr_i=\psi_i$.
 Again
by Lemma \ref{lem normal morphism}, $\gr_i$ is a normal map.
This completes the induction step and the proof of the proposition.
\end{proof}

\noindent
\begin{proof}[\bf Proof of Theorem \ref{thm subnormal factorization}]
Consider the subnormal series of normal maps
\[
\xymatrix{
M_{k+1} \ar[r]^{n_k}\ar@/_3ex/[rrrr]_{n}& M_{k}\ar[r]&\dots\ar[r] &M_2\ar[r]^{n_1} &G=M_1
}
\]
where $k\ge 1, M_{k+1}=M, M_1=G,$ and $n_t$ is a normal map, for all $1\le t\le k$.  Let
$\psi_t\colon\gC\to M_t$ be the maps defined by $\psi_{k+1}=\psi,$ and for $1\le t\le k,$
let $\psi_t=\psi\circ n_k\circ\dots\circ n_t$ (so $\psi_1=\gvp$).

The universality property of the inverse limit and
Proposition \ref{prop same inv lim} imply
that $\gC_{\infty,\psi_t}\cong\gC_{\infty,\psi_{t+1}},$ for all $1\le t\le k,$ so since
$\gC_{\infty,\gvp}=\gC_{\infty,\psi_1}$ and $\gC_{\infty,\psi}=\gC_{\infty,\psi_{k+1}},$
the theorem holds.
\end{proof}
As a corollary to Theorem \ref{thm subnormal factorization} we get

\begin{prop}\label{prop nct}
\begin{enumerate}
\item
Let $H$ be a subnormal subgroup of $G$ containing $\gvp(\gC)$.
Then $\gC_{\infty,\gvp}\cong\gC_{\infty,\psi},$ where $\psi\colon\gC\to H$
is the restriction of $\gvp$ in the range;
\item
if the  series $\{C_i\}$ terminates,
then $\gC_{\infty,\gvp}=\gC_{\infty,\psi},$ where
$\psi\colon\gC\to \calc$
is the restriction of $\gvp$ in the range.
\end{enumerate}
\end{prop}
\begin{proof}
Let $n\colon H\hookrightarrow G$ be the inclusion map.  Then
$n$ is a subnormal map, and $\gC\xrightarrow{\psi} H\xrightarrow{n} G$ is a factorization
of $\gvp$.  Hence (1) follows from Theorem \ref{thm subnormal factorization}, and (2)
is immediate from (1).
\end{proof}

We now turn our attention to the kernel $K$.

\begin{lemma}\label{lem kernel}
Let  $L\le \ker c_{\gvp}$ be a normal subgroup of $\gC$.
Let $\gr\colon\gC/L\to G$ be the homomorphism induced
by $\gvp$, and let $\gb\colon\gC/L\to\cl$ be the homomorphism
induced by $c_{\gvp}$.  Then $\cl$ and $\left(\gC/L\right){^\gr}$
are naturally isomorphic.
\end{lemma}
\begin{proof}
Consider the following commutative diagram. We show that $u$
is an isomorphism whose inverse is $v$.
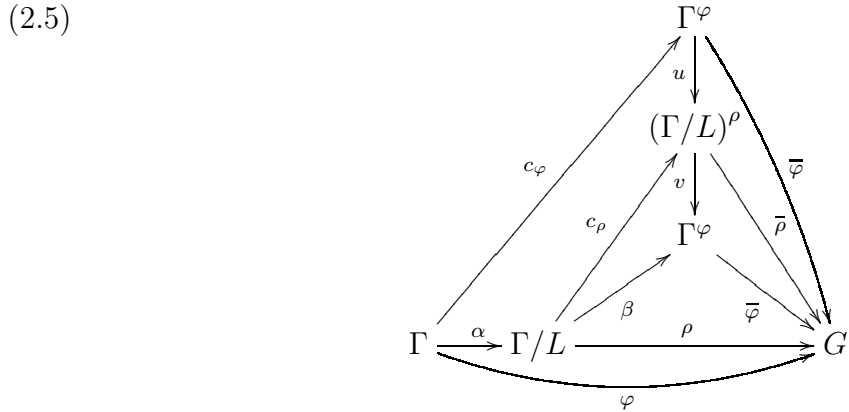
\begin{equation}\label{diag same nc}
\xymatrix{
&&\cl\ar[d]_u\ar@/^/[rddd]^{\W{\gvp}}\\
&&\left(\gC/L\right)^{\gr}\ar[d]_v\ar[rdd]^{\W{\gr}}\\
&&\cl\ar[rd]_{\W{\gvp}}\\
\gC\ar[r]^{\ga}\ar[rruuu]^{c_{\gvp}}\ar@/_3ex/[rrr]_{\gvp}&\gC/L\ar[rr]^{\gr}\ar[ru]_{\gb}\ar[ruu]^{c_{\gr}} &&G
}
\end{equation}
here $\ga\circ\gr=\gvp,\ \ga\circ\gb=c_{\gvp},$ and $u$ and $v$ are the unique normal morphisms
obtained from the universality properties of $\cl$ and $\left(\gC/L\right)^{\gr}$ respectively.

Notice that $\ga\circ\gb\circ\W{\gvp}=c_{\gvp}\circ\W{\gvp}=\gvp=\ga\circ\gr$.  Since $\ga$
is surjective we see that $\gb\circ\W{\gvp}=\gr$.  By the universality property of $(\gC/L)^{\gr}$
we get the map $v$. Also, $(\ga\circ c_{\gr})\circ\W{\gr}=\ga\circ\gr=\gvp$.  By the universality property of $\cl$
we get the map $u$.

Next we have $c_{\gvp}\circ u\circ v=\ga\circ c_{\gr}\circ v=\ga\circ\gb= c_{\gvp}$.
Also $u\circ v\circ\W{\gvp}=u\circ\W{\gr}=\W{\gvp}$.  Hence $u\circ v=1_{\cl},$ by the uniqueness in the universality
property of $\cl$.

Further, $\ga\circ c_{\gr}\circ v\circ u=\ga\circ\gb\circ u=c_{\gvp}\circ u=\ga\circ c_{\gr}$.
Since $\ga$ is surjective,  $c_{\gr}\circ v\circ u=c_{\gr}$.  Also
$v\circ u\circ\W{\gr}=v\circ\W{\gvp}=\W{\gr}$.  So, as above,  $v\circ u$ is the identity map of $\left(\gC/L\right)^{\gr}$.
\end{proof}

As a corollary we get

\begin{cor}\label{cor same nct}
Let  $L\le\ker c_{\gvp}$ be a normal subgroup of $\gC,$
and let $\gr\colon\gC/L\to G$ be the homomorphism induced
by $\gvp$.  Then $\gC/L\xrightarrow{\gb}\cl\xrightarrow{\W{\gvp}} G,$
where $\gb$ and $\W{\gvp}$ are as in diagram \eqref{diag same nc},
is the universal normal decomposition of $\gr$.    
\end{cor}
\begin{proof}
This follows from Lemma \ref{lem kernel}, because in diagram \eqref{diag same nc} we may take
that $c_{\gr}=\gb,$  $(\gC/L)^{\gr}=\cl$ and $\W{\gr}=\W{\gvp}$.
\end{proof}

In particular we get Proposition \ref{prop Kf} of the introduction as part (2) of the following:

\begin{prop}\label{prop Kterm}
Let $L\le\gC$ be a subgroup with $L\le\ker\gvp_i,$ for all integers $i\ge 1$.
Let $\gr\colon\gC/L\to G$ be the map induced by $\gvp,$ let $\{(\gC/L)_i\}=\{(\gC/L)_{i,\gr}\}$
be the normal closures tower of $\gr,$ and let $\gr_i,\,\W{\gr_i}$ be the maps
for the tower $\{(\gC/L)_i\}$. Then
\begin{enumerate}
\item
there is a natural isomorphism $\gC_i\cong\left(\gC/L\right)_i,$ where $\gr_i\colon\gC/L\to (\gC/L)_i$ is the map induced by $\gvp_i,$
and $\W{\gvp_i}=\W{\gr_i},$ for all $i\ge 1$.
In particular $\gC_{\infty}\cong\left(\gC/L\right)_{\infty}$;

\item
if the series $\{K_i\}$ terminates at $\calk,$ then
$(1)$ above holds for $L:=\calk$.
\end{enumerate}
\end{prop}
\begin{proof}
Part (1) is immediate from Corollary \ref{cor same nct}.  For (2)
note that $\gvp_{i+1}(\ker\gvp_i)\le \ker\W{\gvp_i}\le Z(\gC_{i+1})$.  It follows by induction
on $i,$ that $K_i\le\ker\gvp_i$, for all $i\ge 1$. This implies that if
the series $\{K_i\}$ terminates,
then $\calk\le \ker\gvp_i,$ for all $i\ge 1,$
so (2) follows from (1).
\end{proof}

Combining Propositions \ref{prop Kterm}(2) and \ref{prop nct}(2) we get:

\begin{cor}\label{cor frakC and frakK}
Assume that both the series $\{C_i\}$ and the series $\{K_i\}$ terminate at $\calc$ and $\calk$ respectively.  Then
there is an isomorphism $\gC_{\infty,\gvp}\cong(\gC/\calk)_{\infty,\psi},$ where
$\psi\colon\gC/\calk\to \calc,$
is the map induced by $\gvp$.
\end{cor}
\begin{proof}
By Proposition \ref{prop Kterm}(2), $\gC_{\infty,\gvp}=(\gC/\calk)_{\infty,\gr},$ where
$\gr\colon\gC/\calk\to G$ is the map induced by $\gvp$.  Then, by Proposition \ref{prop nct}(2) (with $\gC/\calk$ in place of $\gC$), $(\gC/\calk)_{\infty,\gr}=(\gC/\calk)_{\infty,\psi}$.
\end{proof}

The following theorem  proves in particular the assertion of Theorem \ref{Thm finite}
of the introduction.

\begin{thm}\label{thm finite}
Let $\gvp\colon\gC\to G$ be a homomorphism of finite groups.
Let $\gr\colon\gC\to\calc$ and let $\psi\colon\gC/\calk\to \calc$ be the maps induced by $\gvp$.  Then
\begin{enumerate}
\item
$\gC_{\infty,\gvp}\cong\gC_{\infty,\gr}=(\gC/\calk)_{\infty,\psi}$;

\item
the normal closures series $\{\gC_{i,\gr}\}$ and the series $\{(\gC/\calk)_{i,\psi}\}$ terminate,
hence $\gC_{\infty,\gvp}\cong\gC_{t_1,\gr}\cong (\gC/\calk)_{t_2,\psi},$ where
$\gC_{t_1,\gr}$ and $(\gC/\calk)_{t_2,\psi}$ are  the terminal members of the respective series;

\item
$\gC_{\infty,\gvp}$ is finite and $\vert\gC_{\infty,\gvp}\vert\le |\gC/\calk|\cdot g(|\calc|),$
where $g(1)=1,$ and for any integer $t\ge 2,$
\[
g(t)=t^k,\quad\text{where}\quad k=\half(log_pt + 1)\text{ and $p$ is the least prime divisor of $t$}.
\]
\end{enumerate}
\end{thm}
\begin{proof}
Part (1) follows from Proposition \ref{prop nct}(2) and Corollary \ref{cor frakC and frakK}.  Then, by definition,
$\calc=\lan\gvp(\gC)^{\calc}\ran,$
so parts (2) and (3) follow from \cite[Theorem 4.1]{FS1}.
\end{proof}

Furthermore we can use Corollary \ref{cor frakC and frakK} to prove the following lemma,
which leads to the proof of Theorem \ref{thm nil} of the introduction.

\begin{lemma}\label{lem invlim nil}
Assume that $\gvp(\gC)$ is subnormal in $G$ and that the series $\{K_i\}$ (see Notation \ref{not Ci and Ki}(2)) terminates.
Then $\gC_{\infty}=\gC/\calk.$
\end{lemma}
\begin{proof}
Since $\gvp(\gC)$ is subnormal in $G,$ we have $\calc=\gvp(\gC)$.
By Proposition \ref{prop nct}(2), to evaluate $\gC_{\infty},$
we may assume that $\gvp$ is surjective.
By \cite[Corollary 3.9(2)]{FS1},
$\gC_2=\gC/[\gC,\,K],$ and $\gvp_2\colon\gC\to\gC_2$ is
the canonical map.  Iterating on \cite[Corollary 3.9(2)]{FS1}
we see that $\gC_i=\gC/K_i,$ for all $i\ge 2,$ thus $\gC_{\infty}=\gC/\calk$.
\end{proof}

Part (2) of the following corollary is Theorem \ref{thm nil} of the introduction.
\begin{cor}\label{cor invlim nil}
\begin{enumerate}
\item
If   $\gvp(\gC)$ is subnormal in $G$ 
and $K$ is
contained in $Z_i(\gC)$
for some integer $i\ge 0$
(which holds  if $\gvp$ is injective),
then $\gC_{\infty}=\gC$;

\item
if $\gC$ and $G$ are nilpotent, then $\gC_{\infty}=\gC$.
\end{enumerate}
\end{cor}
\begin{proof}
Part (1) is an immediate consequence of Lemma \ref{lem invlim nil}, since under the
hypotheses of (1), $\calk=1$.  Then (2) follows from (1).
\end{proof}

We now turn to the nilpotent quotients of $\gC_{\infty}=\gC_{\infty,\gvp},$ and
prove Theorem \ref{Thm nil quotients} of the introduction.

\begin{prop}\label{prop nil resid}
Let $k\ge 1$.
The map $\gvp_{(\infty, k)}\colon\gC/\gc_k(\gC)\to\gC_{\infty}/\gc_k(\gC_{\infty})$
induced by the canonical map $\gvp_{\infty}\colon\gC\to\gC_{\infty}$ is injective.
If $\gC$ and $G$ are finite, then  it induces an isomorphism $\gC/\gc_{\infty}(\gC)\cong \gC_{\infty}/\gc_{\infty}(\gC_{\infty})$.
\end{prop}
\begin{proof}
Let $\psi\colon \gC/\gc_k(\gC)\to G/\gc_k(G)$ be the map induced by $\gvp$.  By naturality (Lemma \ref{lem maps between invlimits})
we have the following commutative diagram
\[
\xymatrix{\gC\ar[rr]^{\gvp_{\infty}}\ar[dd]_{\mu}&& \gC_{\infty}\ar[dd]^{\eta_{\infty}}\ar[dl]\ar[r]& G\ar[dd]^{\eta}\\
&\gC_{\infty}/\gc_k(\gC_{\infty})\ar[dr]\\
\gC/\gc_k(\gC)\ar[rr]_{\psi_{\infty}}^\cong\ar[ru]^{\gvp_{(\infty,k)}}&&(\gC/\gc_k(\gC))_{\infty,\psi}\ar[r]&G/\gc_k(G)\\
}
\]
where the map $\eta_{\infty}$ is obtained from Lemma \ref{lem maps between invlimits}, and
where $\psi_{\infty}$ is an isomorphism by Corollary \ref{cor invlim nil}(2).
The diagram is commutative since $\mu$ is surjective.  Since $\psi_{\infty}$ is an isomorphism,
$\gvp_{(\infty,k)}$ is injective.

Assume that $\gC$ and $G$ are finite.  By Proposition \ref{prop nct}(2)
we may assume that $G=\calc,$ so $G=\lan\gvp(\gC)^G\ran$.  Also, by
Theorem \ref{thm finite}(2), the series $\{\gC_i\}$ terminates, so $\gC_{\infty}=\gC_t,$ for some
$t\ge 1,$ and now $\gvp_{\infty}=\gvp_t$.  By \cite[Lemma 4.2(1)]{FS1}, $\gC_{\infty}=\gC_t=\lan\gvp_t(\gC)^{\gC_t}\ran$.
Hence the conjugates of the image of $\gvp_t(\gC)$
in $\gC_t/\gc_k(\gC_t)$ generates it. By  \cite[Lemma 4.4]{FS1},
the image of $\gvp_t(\gC)$ in $\gC_t/\gc_k(\gC_t)$ equals $\gC_t/\gc_k(\gC_t)$.
This shows that $\gvp_{(\infty, k)}$ is surjective.
\end{proof}

\section{Universality of $\gC_\infty$}\label{sect universal}
In this section we  prove theorem \ref{thm universal}
of the introduction.

\begin{proof}
To facilitate the discussion, and in view of the theorem  we are now proving, we refer to the
factorization $\gC\to \gC_{\infty,\gvp}\to G $
as the {\it subnormal closure of $\gvp.$}  When the maps are understood, we refer to the
group  $\gC_{\infty,\gvp}$ itself as the subnormal closure of $\gC$ with respect to  $G.$
Notice that the subnormal closure is a functor  on maps of groups  and thus  acts on squares as in equation \eqref{eq naturality} of   
Lemma \ref{lem maps between invlimits}, and  respects compositions. We use the following

%
\begin{notation}
Given a commutative diagram of group homomorphisms:
\[
\xymatrix{
\gC\ar[r]^{\gvp}\ar[d]^= & G\ar[d]^{\eta}\\
\gC\ar[r]^{\psi} &H\\
}
\]
we denote by $\bar \eta_\infty$ the induced map on the subnormal closures:
\[
\bar \eta_\infty: \gC_{\infty,\gvp}\to  \gC_{\infty,\psi}.
\]
\end{notation}

\noindent
Throughout this proof $\gC_{\infty}$ denotes $\gC_{\infty,\gvp}$.
We begin the proof by noting that by Theorem \ref{thm finite}, $\gC_{\infty}$ is finite.
Further, by Theorem \ref{thm finite}, the map  $\gr\colon\gC\to\calc,$ where $\calc$ is as in Notation \ref{not Ci and Ki}(1),
satisfies  $\gC_{\infty}=\gC_{t,\gr}=\gC_{\infty,\gr},$ for some integer $t\ge 1$. By the same theorem we have a finite tower of normal maps, the first $r$ being normal inclusions leading from the subnormal closure $\calc$ to $G:$
\begin{equation}\label{eq finite tower}
\gC_{\infty}=\gC_{t,\gr}\xrightarrow{\W{\gr_{t-1}}}\gC_{t-1,\gr}\cdots \to\gC_{2,\gr}\to \calc= C_r \hookrightarrow C_{r-1}\hookrightarrow\dots\hookrightarrow G.
\end{equation}
This implies that  the canonical map:
\[
l: \gC_{\infty}\to G
\]
is a subnormal map,  by definition.  But by lemma 4.2 of \cite{FS1} the normal closure of the image of
the map $\gvp_\infty:\gC\to \gC_{\infty}=\gC_{t,\gr},$ is $\gC_{\infty}$.
Since the tower \eqref{eq finite tower} is finite and terminates at $\gC_{\infty}$,
we have 
\[
\gC^{\gvp_\infty}\to \gC_{\infty}
\]
is the trivial extension and we take this map  as the identity.
Of course this implies that the normal closures tower of $\gvp_{\infty}\colon\gC\to\gC_{\infty}$
is constant, so
\[
\bar l_\infty:\gC_{\infty,\gvp_\infty}\longrightarrow \gC_{\infty}
\]
is  the identity map.

Now we show that the map $l$ is initial  among all
subnormal factorization  maps of $\gvp$.   Let $\gC\xrightarrow{\psi}  S\xrightarrow{s} G$ be a factorization of our map $\gvp$ via  a subnormal map $s:S\to G$. We need to show that there is a unique
map $\tilde s: \gC_\infty \to S$ rendering the following diagram commutative.
\begin{equation}\label{eq tilde s}
\xymatrix{
\gC\ar[r]^{\gvp_\infty}\ar[d]^=&\gC_{\infty}\ar[r]^{l}\ar[d]^{\tilde s} & G\ar[d]^=\\
\gC\ar[r]^\psi &S\ar[r]^s &G\\
}
\end{equation}

To see this we consider the map induced  on the subnormal closures by the given subnormal map $s:$
\[
\xymatrix{
\gC\ar[r]\ar[d]^=&\gC_{\infty,\psi}\ar[r]^\lambda\ar[d]_\cong^{\bar s_\infty} & S\ar[d]^s\\
\gC\ar[r]&\gC_{\infty} \ar[r]^l &G\\
}
\]
Here $\bar s_\infty$ is the map induced by naturality of the subnormal closure  as in Lemma \ref{lem maps between invlimits}, and
$\gC_{\infty,\psi}$ is the subnormal closure  of the map $\psi: \gC\to S.$  Since  $\bar s_\infty$ is an isomorphism by 
Proposition \ref{prop same inv lim},
one gets a well defined map
\[
\tilde s=\lambda\circ ({\bar s}_\infty)^{-1} : \gC_{\infty}\to S.
\]

To see that this latter map is unique consider any map   $\sigma: \gC_{\infty}\to S$
as in diagram \ref{eq tilde s}, with $\tilde s$ replaced with $\gs$.
This map $\sigma$ induces, by naturality, the following commutative  diagram of groups, where the  two lower squares
do not depend on the choice of $\sigma$:

\[
\xymatrix{
\gC\ar[r]\ar[d]^=&\gC_{\infty,{\gvp_\infty}}\ar[r]^=\ar[d]^{\bar{\sigma}_\infty} & \gC_{\infty}\ar[d]^\sigma\ar[r]^l& G\ar[d]^=\\
\gC\ar[r]\ar[d]^=&\gC_{\infty,\psi}\ar[r]^\lambda\ar[d]^{\tilde s_\infty} &S\ar[d]^s\ar[r]^s&G\\
\gC\ar[r]&\gC_{\infty} \ar[r]^l &G\\
}
\]
Now we  rewrite the map $\sigma $ in terms of the maps in the given decomposition $\gC\to S \to G$ alone: We read in the middle upper square: $\sigma=\bar \sigma_\infty\circ \lambda.$
But since:
\[
\bar\sigma_\infty\circ\bar s_\infty=\widebar{(\gs\circ s)}_\infty = \widebar{{l}}_\infty=id
\]
It follows that both $\bar\sigma_\infty,$  being the inverse to $\bar s_\infty,$ and thus $\sigma$ itself, are determined by $\lambda$---constructed out of $\psi$ and
 $s$ as claimed.
\end{proof}

\section{Examples}

In this section we give two examples.  In both examples
we take $G$ to be perfect with $G=\lan\gvp(\gC)^G\ran$.
In the first Example \ref{eg G perfect} we assume that $\gC$ and $G$ are finite and that $\ker\gvp\le Z_{\infty}(\gC),$
and we show that $\ker\gvp_{\infty}$ can have arbitrarily large nilpotency class.
In the second Example \ref{eg gC perfect} we show that if $\gC$ is perfect, then $\gC_{\infty}=\gC_2$ is
the universal $\gvp$-central extension of $G$ (see \cite{FS2}).

\begin{example}\label{eg G perfect}

In this example we  assume that $\gC$ and $G$ are finite and
that $\ker\gvp\le Z_{\infty}(\gC)$.
The purpose of this example is to show that the nilpotency class
of $\ker\gvp_{\infty}$ can be arbitrarily large.
We first need
the following easy lemma and its corollary.

\begin{lemma}\label{lem G perfect}
Let $G$ be perfect and set $H:=\gC_{\infty}$.  Then $H$ is a central
product\linebreak $H=H^{(\infty)}\circ Z_{\infty}(H)$.
Here, $H^{(\infty)}$ is the last term of the derived series of $H$.
If the nilpotent residual $\gC/\gc_{\infty}(\gC)$ of $\gC$
has nilpotency class $c,$ then the nilpotency class of $Z_{\infty}(H)$
is $\le c+1$.
\end{lemma}
\begin{proof}
First, by \cite[Theorem 4.1]{FS1}, $H=\gC_{\infty}$ is finite, because $G=\lan\gvp(\gC)^G\ran$.
Set $L=H^{(\infty)}$ and $\calz=Z_{\infty}(H)$.
There is   surjection $H\onto G$ whose kernel is contained in $\calz,$
and since $G$ is perfect, we get that $H=L\calz$.
Since $L$ is perfect, $Z_{\infty}(L)=Z(L),$ and of course $[L, \calz]\le Z_{\infty}(L)=Z(L)$.
By the three subgroup lemma we get $[L,\calz]=1,$ and so $H=L\circ \calz$.

Next, by Proposition \ref{prop nil resid}, the nilpotent residuals of $\gC$ and $H$ are isomorphic.
Clearly $\gc_{\infty}(H)=L$ and $\gC/\gc_{\infty}(\gC)\cong H/\gc_{\infty}(H)=H/L\cong \calz/(\calz\cap L)$ is nilpotent of class $c$.
Hence, since $\calz\cap L\le Z(\calz),$ we see that the nilpotency class of $\calz$
is $\le c+1$.
\end{proof}
\begin{cor}\label{cor G perfect}
Let $G$ be perfect, and let $c$ (resp.~$d$) be the nilpotency class
of the nilpotent residual of $\gC$ (resp.~of $K:=\ker\gvp$).
Then $\ker\gvp_{\infty}$ has nilpotency class $\ge d-c-1$.
\end{cor}
\begin{proof}
Let $f\colon \gC_{\infty}\onto G$.  Then $\ker f\le Z_{\infty}(\gC_{\infty})$.
Since $\gvp_{\infty}\circ f=\gvp,$ we see that $\gvp_{\infty}(K)\le\ker f\le Z_{\infty}(\gC_{\infty})$.
Now $\ker\gvp_{\infty}\le K$.  By Lemma \ref{lem G perfect}, the nilpotency class of $\gvp_{\infty}(K)\le c+1$.
Since $K/\ker\gvp_{\infty}\cong \gvp_{\infty}(K),$ and the nilpotency class
of $K$ is $d,$ the corollary follows.
\end{proof}
\medskip

We now construct examples where $G$ is perfect (in fact simple), $\gc_{\infty}(\gC)=[\gC,\gC],$ so one has:
\begin{itemize}
\item
The nilpotent residual of $\gC$ has nilpotency class $c=1$.

\item
The nilpotency class $d$ of $\ker\gvp$ is arbitrarily large.
\end{itemize}

\noindent
By Corollary \ref{cor G perfect}, the nilpotency class of $\ker\gvp_{\infty}$
is $\ge d-2,$ so it is arbitrarily large.

Let $p$ be a prime, let $n\ge 2$ such that $p$ divides $n$ and let $q$ be a prime power
such that $p$ divides $q-1$.
Let $H$ be an image of ${\rm SL}_n(q)$ such that $Z(H)\cong\zz_p$ is cyclic of order $p$.
Let $B$ be an elementary abelian group of order $p^n,$
and let $\gC$ be the wreath product $G=H\,wr\, B$.
By \cite[section 3, p.~282--283]{M}, $Z_{\infty}(\gC)$ is of nilpotency
class $(n-1)p+1$ (\cite[Lemma 3.2, p.~283]{M}).
Also, $\gc_{\infty}(\gC)=[\gC,\gC]$ is isomorphic to a
direct product of $p^n$ copies of $H$ (\cite[Lemma 3.1, p.~283]{M}).
Clearly $\gC/Z_{\infty}(\gC)$ is isomorphic to ${\rm PSL}_n(q)\, wr\, B$.
But ${\rm PSL}_n(q)\, wr\, B$ is contained in $G={\rm PSL}_{np^n}(q)$.
It follows that there exists a homomorphism $\gvp\colon\gC\to G$ whose
kernel is $Z_{\infty}(\gC)$.  Hence $G, \gC$ and $\ker\gvp$ have
the claimed properties.
\end{example}

\begin{example}\label{eg gC perfect}
Suppose $\gC$ is perfect.
Then  under the above assumptions, $G$ is perfect and $\gC_{\infty}=\gC_2$ is a perfect group which is the universal
$\gvp$-central  extension of $G$.

Indeed, since $G=\lan \gvp(\gC)^G\ran,$ and $\gvp(\gC)$ is perfect (because $\gC$ is),
$G$ is perfect.
Also, by \cite[Lemma 3.3]{FS1}, $\gC_2$ is a central extension of $G,$
and $\gC_2$ is generated by $\{\gvp(\gC)^g\mid g\in G\},$ so $\gC_2$
is perfect.  The same argument shows that $\gC_3$ is a perfect
central extension of $\gC_2$.  By \cite[Prop.~1.8, p.~637]{CDFS}, $\gC_3$
is a central extension of $G$.  It is easy to check now that $\gC_2=\gC_3,$
by the universal property of $\gC_2,$  and that the assertion above holds.
\end{example}



\begin{thebibliography}{99}



\bibitem[BH]{BH} R.~Brown, P.~J.~Higgins, 
{\it On the connection between the second relative homotopy groups of
some related spaces,} Proc.~London Math.~Soc.~{\bf 36} (1978), 193--212.

\bibitem[BHS]{BHS} R.~Brown, P.~J.~Higgins, R.~Sivera, {\it Nonabelian algebraic topology,} 
EMS Tracts in Mathematics, {\bf 15}. European Mathematical Society (EMS), Z\"urich, 2011.

\bibitem[BK]{BK} A.~Bousfiled, D.~Kan, {\it Homotopy limits, completions and localizations} Springer LN 304, 1972.


\bibitem[CDFS]{CDFS} W.~Chach{\'o}lski, E.~Damian, E.D.~Farjoun, Y.~Segev.  {\it The A-core and A-cover of a group,}
J.~Algebra {\bf 321}  (2009),  no.~2, 631--666.


\bibitem[FS1]{FS1} E.~D.~Farjoun, Y.~Segev, {\it Normal closure and injective normalizer
of a group homomorphism,} submitted, 2014 (http://arxiv.org/abs/1403.3501).

\bibitem[FS2]{FS2} E.~D.~Farjoun, Y.~Segev,
{\it Relative Schur multiplier and the universal extensions of group homomorphisms,}
preprint, 2014.


\bibitem[M]{M} J.~D.~P.~Meldrum, {\it On central series of a  group,} J.~Alg.~{\bf 6} (1967), 281--284.

\bibitem[R]{R} D.J.S.~Robinson, {\it A Course in the Theory of Groups,} Graduate Texts in Mathematics. Springer-Verlag,
1996.
\end{thebibliography}
\end{document}